\documentclass[a4paper,12pt]{amsart}

\usepackage{amsmath,amssymb,amsthm}    
\usepackage{comment}
\usepackage{ulem}
\usepackage{tikz}
\usepackage{mathdots}
\usepackage{tikz}
\usepackage{subcaption}
\usepackage[subrefformat=parens,labelformat=parens]{subcaption}
\usetikzlibrary{calc,positioning,decorations.pathmorphing,decorations.pathreplacing}
\tikzset{emp/.style={double distance = 0.3ex}}

\tikzset{M edge/.style={line width=1.3pt,double distance=1.1pt}}
\tikzset{F1 edge/.style={line width=1.3,color=red,->}}
\tikzset{F2 edge/.style={line width=1.3,color=blue,->}}
\tikzset{E edge/.style={line width=1.3,color=black,-}}

\tikzset{squared black vertex/.style={draw,minimum size=2mm,inner sep=0pt,outer sep=3pt,fill=black, color=black}}
\tikzset{red vertex/.style={circle,draw,minimum size=2mm,inner sep=0pt,outer sep=2pt,fill=red, color=red}}
\tikzset{blue vertex/.style={circle,draw,minimum size=2mm,inner sep=0pt,outer sep=2pt,fill=blue, color=blue}}
\tikzset{black vertex/.style={circle,draw,minimum size=2mm,inner sep=0pt,outer sep=2pt,fill=black, color=black}}
\tikzset{small black vertex/.style={circle,draw,minimum size=1.2mm,inner sep=0pt,outer sep=1.2pt,fill=black, color=black}}
\tikzset{small white vertex/.style={circle,draw,minimum size=1.2mm,inner sep=0pt,outer sep=1.2pt,color=black,fill=white}}
\tikzset{square vertex/.style={draw,minimum size=1.2mm,inner sep=0pt,outer sep=1.2pt,fill=black, color=red}}
\tikzset{white vertex/.style={circle,draw,minimum size=2mm,inner sep=0pt,outer sep=3pt,color=black,fill=white}}

\tikzset{fatpath/.style={line width=9pt,rounded corners=.1mm}}

\tikzstyle{edge}=[line width=1.3]

\tikzstyle{color1}=[color=blue] 
\tikzstyle{color2}=[color=red]
\tikzstyle{color3}=[color=green] 
\tikzstyle{color4}=[fill=yellow]
\tikzstyle{color5}=[ dashed] 

\tikzstyle{backcolor1}=[color=gray!55!white] 
\tikzstyle{backcolor2}=[color=blue!35!white] 

\usepackage{standalone}
\usepackage{mathtools}

\usepackage{enumitem}

\usepackage{amsaddr}

\usepackage[T1]{fontenc}
\usepackage[english]{babel}

\usepackage{fullpage}

\newtheorem{theorem}             {Theorem}
\newtheorem{lemma}     	[theorem] {Lemma}        
\newtheorem{conjecture}	[theorem] {Conjecture}

\newtheorem{proposition}[theorem] {Proposition}

\newtheorem{claim}	[theorem] {Claim}  

\newcommand{\id}{\iota}

\newcommand{\eps}{\varepsilon}
\newcommand{\calF}{\mathcal{F}}

\title{Independent dominating sets \\ in planar triangulations} 
\thanks{This research has been partially supported by Coordenação de Aperfeiçoamento
  de Pessoal de Nível Superior -- Brazil -- CAPES -- Finance Code 001. 
  F.~Botler is supported by CNPq (Proc.~423395/2018-1 and~304315/2022-2), 
  FAPERJ (Proc.~211.305/2019 and~201.334/2022) and 
  CAPES-PRINT (Proc.~88887.695773/2022-00). 
  C.G. Fernandes was partially supported by CNPq (Proc.~308116/2016-0, 423833/2018-9, and~310979/2020-0) and FAPESP (Proc.~2019/13364-7).
  FAPERJ and FAPESP are, respectively, Research Foundations of Rio de Janeiro
  and S\~ao Paulo, Brazil.  CNPq is the National Council for Scientific and
  Technological Development of Brazil.
  J. Gutiérrez was partially supported by
by Movilizaciones para Investigación AmSud, PLANarity and distance IN Graph theory (E070-2021-01-Nro.6997) and Fondo Semilla UTEC 871075-2022.  
  }

\author{Fábio Botler}
\address{\vspace{-4mm}Programa de Engenharia de Sistemas e Computa\c c\~ao \\
Instituto Alberto Luiz Coimbra de P\'os-Gradua\c c\~ao e Pesquisa em Engenharia \\
Universidade Federal do Rio de Janeiro, Brazil}
\email{fbotler@cos.ufrj.br}

\author{Cristina G. Fernandes}
\address{\vspace{-4mm}Departamento de Ci\^encia da Computa\c c\~ao \\
Universidade de S\~ao Paulo, Brazil}
\email{cris@ime.usp.br}

\author{Juan Gutiérrez} 
\address{\vspace{-4mm}Departamento de Ciencia de la Computación\\ 
Universidad de Ingeniería y Tecnología (UTEC), Perú}
\email{jgutierreza@utec.edu.pe}

\usepackage{lineno}

\begin{document}
\normalem

\maketitle

\begin{abstract}
  In 1996, Matheson and Tarjan proved that every near planar triangulation on \(n\) vertices contains 
  a dominating set 
  of size at most \(n/3\), and conjectured that this upper bound can be reduced to \(n/4\) 
  for planar triangulations when $n$ is sufficiently large.
  In this paper, we consider the analogous problem for independent dominating sets:
  What is the minimum \(\eps\) for which
  every near planar triangulation on \(n\) vertices contains an independent dominating set of size at most \(\eps n\)?
  We prove that \(2/7 \leq \eps \leq 5/12\).
  Moreover, this upper bound can be improved to $3/8$ for planar triangulations, and to \(1/3\) for planar triangulations with minimum degree 5.
\end{abstract}

\section{Introduction}

Let \(S\) be a set of vertices of a graph \(G\).
We say that \(S\) is \emph{dominating}  
if each vertex of~$G$ is in~$S$ or is a neighbor of some vertex in~$S$;
and that \(S\) is \emph{independent} if no two vertices in~$S$ are adjacent.
In particular, any maximal independent set in~$G$ is dominating.  
We denote by~$\gamma(G)$ (resp. \(\id(G)\)) the cardinality of a minimum dominating (resp. independent dominating) set of~$G$.
Note that~$\gamma(G) \leq \id(G)$. 
Such parameters are known as the \emph{domination number} and the \emph{independent domination number}
of \(G\), respectively, and their calculations are known to be NP-hard problems 
even on planar bipartite graphs with maximum degree \(3\)~\cite[Corollary~3]{ZvervichZ1995}.
Therefore, it is natural to explore such parameters in special classes of graphs, 
or to look for upper and lower bounds on them.

In this paper, we focus on \emph{planar} graphs,
i.e., graphs that can be drawn in the plane so that intersections of edges happen only at their ends.
By a \emph{plane} graph we mean a planar graph together with a fixed planar drawing of it.
For general terminology on planar graphs we refer to the book of Diestel~\cite{Diestel2010}.
In particular, a \emph{planar triangulation} is a plane graph in which each face is bounded by a triangle; 
and a \emph{triangulated disk} or a \emph{near (planar) triangulation} is a 2-connected plane graph 
in which each face, except possibly its outer face, is bounded by a triangle.
In 1996, Matheson and Tarjan~\cite{MathesonT1996} proved that every near triangulation~\(G\) on~\(n\) vertices satisfies \(\gamma(G) \leq n/3\), and posed the following conjecture.

\begin{conjecture}[Matheson--Tarjan, 1996]\label{conj:mathesontarjan}
  For every sufficiently large planar triangulation~$G$ on~$n$ vertices,
  we have ${\gamma(G) \leq n/4}$.
\end{conjecture}

Conjecture~\ref{conj:mathesontarjan} is best possible since, for any planar triangulation~\(G\) 
that consists of a triangulation of~\(k\) vertex-disjoint copies of~\(K_4\) embedded so that 
each copy of~\(K_4\) is in the outer face of all other copies of~\(K_4\), 
we have \(\gamma(G) = k = n/4\).
In spite of the efforts of specialists~\cite{KingP2010,LiZSX2016,Plummer2016,Plummer2020},
the best general upper bound found so far for \(\gamma(G)\) is due to \v{S}pacapan~\cite{Spacapan2020},
who proved that $\gamma(G) \leq 17n/53$ 
for every planar triangulation~$G$ on~$n \geq 6$ vertices.
Observe also that the bound of Conjecture~\ref{conj:mathesontarjan}
is not satisfied for near triangulations
\cite[Figure~1]{MathesonT1996}.
Related results for maximal outerplanar graphs have been given
in \cite{Campos2013,Tokunaga2013}.

We are interested in the analogous problem for the independent domination number:
What is the minimum \(\eps\) such that 
\(\id(G) \leq \eps n\) for every near triangulation \(G\) on \(n\) vertices?
In contrast to the domination number, this parameter has not received so much attention on near triangulations.
It is known that $\id (G) < 3n/4$ for any connected planar graph~$G$ on $n\geq 10$ vertices~{\cite[Theorem~6]{GoddardH2020}}; and that $\id (G) \leq~n/2$ whenever $G$ is planar and $\delta(G)\geq 2$
\cite[Theorem~8]{GoddardH2020}.
Also, if $G$ is a 2-connected outerplanar graph
on $n\geq 5$ vertices, then $\id(G)\leq \frac{2n+1}{5}$ \cite[Theorem 1]{Goddard2023}.
For an excellent survey on independent dominating sets,
see~\cite{Goddard2013}.

Now, note that since every Eulerian planar triangulation has chromatic number \(3\),
they contain three disjoint independent dominating sets.
Goddard and Henning \cite[Question~1]{GoddardH2020}
asked whether such three sets exist in any planar triangulation.
In particular, this would imply that $\id (G) \leq n/3$ for every \(n\)-vertex planar triangulation $G$.
We state the later statement as a conjecture.

\begin{conjecture}\label{conj:planartr}
  For every planar triangulation~$G$ on~$n$ vertices, we have~$\id (G) \leq n/3$.
\end{conjecture}

Our main contribution is the following theorem.

\begin{theorem}\label{thm:mainID}
For every near triangulation $G$ on $n$ vertices, we have $\id(G) \leq  5n/12$. 
Moreover, if $G$ is a planar triangulation, then $\id(G) < 3n/8$, 
and if $G$ is a planar triangulation with $\delta(G)=5$, then $\id(G) \leq n/3$. 
\end{theorem}

We also show that the first two upper bounds are not less than \(2n/7\), 
by presenting an infinite family of planar triangulations \(G\)
for which  \(\id(G) \geq 2n/7\) (see Theorem~\ref{thm:lowerBoundID}).
Note that this improves an observation of 
Goddard and Henning \cite[Figure~6]{GoddardH2020},
who presented an infinite family of planar triangulations \(G\) for which  \(\id(G) \geq 5n/19\).

As a starting example, we prove that \(\id(G)\leq 2n/5\) 
for every planar triangulation~\(G\) on~\(n\) vertices due to 
a relation between \(r\)-dynamic and acyclic colorings as follows.
A \emph{$k$-coloring} of a graph~$G$ is a partition of~$V(G)$ into~$k$ independent sets. 
Each part in such a partition is called a \emph{color class}.
A coloring of~$G$ is \emph{$r$-dynamic} if each vertex~$v$ has neighbors in at 
least~$\min\{r,d(v)\}$ color classes, where~$d(v)$ denotes the degree of~$v$ in~$G$;
and a coloring of \(G\) is \emph{acyclic} if the union of any two of its color classes 
induces a forest.
We use the following result of Goddard and Henning~\cite[Lemma~4]{GoddardH2020}.

\begin{lemma}[Goddard--Henning, 2020]\label{lem:God}
  For every graph~$G$ on~$n$ vertices with~$\delta(G)\geq r$ 
  for which there is an~$r$-dynamic~$k$-coloring, we have~$\id (G) \leq (k-r)\,n/k$.
\end{lemma}

Borodin~\cite{Borodin1979} showed that every planar graph admits an acyclic \(5\)-coloring.
Let $G$ be a planar triangulation and \(\chi_a\) be an acyclic \(5\)-coloring of $G$.
If $G$ is a triangle, clearly $\id(G) \leq 2n/5$. 
Otherwise, \(\delta(G)\geq 3\) and the neighborhood of each vertex contains a cycle;
hence, because~\(\chi_a\) is an acyclic coloring, every vertex has neighbors in at least three 
color classes of \(\chi_a\). 
Therefore, \(\chi_a\) is \(3\)-dynamic and, by Lemma~\ref{lem:God}, we have~$\id(G) \leq 2n/5$.

\section{An improved upper bound}

In this section, we prove the main theorem of this paper, Theorem~\ref{thm:mainID}.
For that, we introduce a concept and settle some notation. 
For a vertex~$v$ in \(G\), denote by~$N(v)$ the set of neighbors of~$v$ in~\(G\). 
For a vertex set~$S$ of~\(G\), denote by~$N(S)$ the set of all neighbors of vertices in~$S$
(that may also include vertices  of \(S\)),
and denote by~$N[S]$ the closed neighborhood of~$S$, that is,~$N[S] = N(S) \cup S$.

Recall that a \textit{near triangulation} is a 2-connected planar graph in which any face is bounded by a triangle, except possibly for the outer face,
and note that every planar triangulation is a near triangulation.
We start with some basic properties of near triangulations.
In the next propositions, we consider a near triangulation $G$ on $n\geq 4$ vertices 
with a fixed planar embedding and let $O$ be the boundary of the outer face of $G$.

\begin{proposition}\label{prop:neigh-vertexO}
Let \(u\in V(G)\) and let \(G' = G[N(u)]\).
Then \(G'\) contains a spanning cycle unless \(u\in V(O)\) and \(O\) is not a triangle,
in which case \(G'\) contains a spanning path.
\end{proposition}

For the next two propositions, let $S$ be an independent set in $G$,
and put $H=G-S$.
Since the neighborhood of every vertex of \(S\) in \(G\) contains a spanning path,
any path of~\(G\) that joins two vertices of \(H\) can be modified to a path in~$G$ that avoids \(S\).
This yields the following result.

\begin{proposition}\label{prop:H-conn}
$H$ is connected.
\end{proposition} 

Now, let $Y$ be the vertices of $S$ in $O$ and put $X = S \setminus Y$.

\begin{proposition}\label{prop:outsideH}
  Vertices of $Y$ lie in the outer face of $H$. 
  Vertices of $X$ lie in inner faces of $H$.
  Moreover, each inner face of $H$ contains at most one vertex of $X$.
\end{proposition}

Now, given a plane graph \(H\), we denote by $f_i(H)$ the number of faces of degree $i$ in $H$,
that is, faces with $i$ vertices in their boundary.

\begin{lemma}\label{lemma:faces}
  For every connected plane graph $H$,
  $f_4(H) + 2 \sum_{i\geq 6} f_i(H) \leq |V(H)|-2$.
\end{lemma}
\begin{proof}
For every $i$, let $f_i =f_i(H)$.
Note that,~$\sum_{i \geq 3} i\,f_i = 2|E(H)|$ and the number of faces of $H$ is $\sum_{i \geq 3} f_i$. Thus, by Euler's formula (e.g.,~\cite[Theorem~4.2.9]{Diestel2010}), we have 
$$\sum_{i \geq 3} f_i + |V(H)| \ = \ \frac12 \sum_{i \geq 3} i f_i + 2.$$
Hence
\begin{align*}
|V(H)| - 2 & \  =  \ \frac12 \sum_{i \geq 3} i f_i - \sum_{i \geq 3} f_i \\
           & \  =  \ \sum_{i \geq 3} \frac{i-2}2 f_i \\ 
           & \ \geq \ \sum_{i \geq 3} \frac{i-2}2 f_i - \frac12 f_3-\frac32 f_5 \\ 
           & \ \geq \ f_4 + 2 \sum_{i\geq 6} f_i. \qedhere
\end{align*}
\end{proof}

Although not needed here,
we observe that Lemma \ref{lemma:faces} can be strengthen
to the statement $f_4(H) + 2 \sum_{i\geq 6} f_i(H) \leq |V(H)|-2-\frac{f_3(H)+3f_5(H)}{2}$.

\begin{proof}[Proof of Theorem~\ref{thm:mainID}]
  Let~$G$ be a near triangulation on~$n$ vertices.
  The celebrated Four-Color Theorem assures that there exists a 
  4-coloring for~$G$ (e.g.,~\cite[Theorem 5.1.1]{Diestel2010}).
  Let~$C_1, C_2, C_3, C_4$ be the color classes in such a coloring. 
  Each~$C_i$ is an independent set. 
  If some~$C_i$ is empty, then each of the other color classes is non-empty 
  and dominating, because each vertex is in a triangle. 
  Therefore, the smallest of the three non-empty color classes 
  is an independent dominating set of size at most~$n/3 < 3n/8 \leq 5n/12$.

  So suppose each~$C_i$ is non-empty and note that $n \geq 4$.  
  For each~$i$, let~$U_i$ be the set of vertices that are not dominated by~$C_i$.  
  Note that~$U_1$,~$U_2$,~$U_3$, and~$U_4$ are pairwise disjoint, 
  as the neighborhood of any vertex \(u\) is colored with at least two colors, 
  distinct from the color used in \(u\). 
  We start by proving a stronger statement on~$U_1$,~$U_2$,~$U_3$, and~$U_4$, namely that
  \begin{eqnarray}
    N[U_i] \cap U_j & = & \emptyset \mbox{\quad\quad if~$i \neq j$.} \label{eq:NUdisjU}
  \end{eqnarray}
  Indeed, by contradiction, say~$v \in N(U_1) \cap U_2$. 
  Then~$v \in C_3 \cup C_4$.  Let~$u$ be a neighbor of~$v$ in~$U_1$.  
  As~$v \in U_2$, we conclude that~$u \in N(U_2) \cap U_1$, 
  and hence~$u \in C_3 \cup C_4$ also.
  Because~$G$ is a near triangulation, it is 2-connected. As~$u$ and~$v$ are adjacent, 
  $u$ and~$v$ have a common neighbor, say~$w$. 
  Since either~$v \in C_3$ and~$u \in C_4$, or~$v \in C_4$ and~$u \in C_3$, 
  $w$ has either color 1 or 2, contradicting the fact that~$v \in U_2$ and~$u \in U_1$.

  Let~$S_i$ be an independent dominating set of~$G[U_i]$ and let~$S = S_1 \cup S_2 \cup S_3 \cup S_4$. 
  Because each~$S_i$ is a subset of~$U_i$, 
  by~\eqref{eq:NUdisjU}, set~$S$ is independent in~$G$.
  In order to use Propositions~\ref{prop:neigh-vertexO},~\ref{prop:H-conn}, and~\ref{prop:outsideH}, 
  let $O$ be the boundary of the outer face of $G$,
  let $Y$ be the vertices of $S$ in $O$, and let $X = S \setminus Y$. 
  Finally, let~$H = G-S$, $n'=|V(H)|$, and consider the plane embedding of~$H$ 
  induced by $G$, which is connected by Proposition \ref{prop:H-conn}.
  
  \begin{claim}\label{claim:2|X|leqn'-2+f_4}
    $2\,|X| \leq n'-2+f_4(H)$.
  \end{claim}
  \begin{proof}
  By Proposition \ref{prop:outsideH}, any vertex of $X$ is in an inner face of $H$,
  and each such face has at most one vertex of $X$.  
  Now, let \(u \in X\),
  and let \(F\) be the face of \(H\) that contains \(u\).
  Since \(u\in U_i\) for some \(i\), 
  only two colors can appear in the vertices in the boundary of \(F\),
  and hence the number of vertices in the boundary of $F$ is even. 
  Thus,
  $$ 2\,|X| \ \leq \ 2\sum_{i \text{ is even}}f_i(H) \ \leq \ 2f_4(H) + 2\sum_{i\geq 6}f_i(H) 
            \ \leq \ n'-2+f_4(H),$$
  where the last inequality follows from Lemma \ref{lemma:faces}.
  \end{proof}
  
  Now, because $|S|=|X|+|Y|$, by Claim \ref{claim:2|X|leqn'-2+f_4}, we have
  \begin{equation}\label{eq:3|X|+|Y|leqn-2+f_4}
    3|X|+|Y| \ = \ 2|X|+|S| \ \leq \ n'-2+f_4(H)+|S| \ = \ n-2+f_4(H).
  \end{equation}
  
  In what follows, we use $|O|$ to denote the number of vertices in~$O$.
  Because \(Y\) is an independent set and \(O\) is a cycle, we have $|Y| \ \leq \ |O|/2$.
  Thus, from~\eqref{eq:3|X|+|Y|leqn-2+f_4}, we deduce that 
  \begin{equation}
  3|S| \ = \ 3|X| + 3|Y| \ \leq \ n-2+f_4(H)+|O|.
  \end{equation}

  By the definition of \(S_i\), each set~$C_i \cup S_i$ is an independent dominating set. 
  Moreover,
  \[\sum_{i=1}^4 (|C_i| + |S_i|) \ = \ n + |S| \ \leq \ \frac{4n-2+f_4(H)+|O|}3.\]
  Therefore, the smallest of these four independent dominating sets has size 
  at most ${n/3+(f_4(H)+|O|-2)/12}$.
  Let \(f'_4\) be the number of inner faces of degree~4 of \(H\),
  and let~$X_4$ be the set of vertices of $S$ in such inner faces.
  Note that \(f_4(H) \leq f'_4 + 1\) because the outer face may have degree~4.
  Since \(G\) is a near triangulation,
  there is exactly one vertex of \(X_4\) in each such inner face,
  so \(f'_4 = |X_4|\).
  By Proposition~\ref{prop:outsideH}, the vertices in $X_4$ are in $X$, 
  and hence not in~$O$. So \(|X_4| + |O| \leq n\) and 
  $f_4(H)+|O| \leq f'_4 + 1 + |O| = |X_4| + 1 +|O| \leq n + 1$.
  Thus, $\id(G) \leq n/3+(f_4(H)+|O|-2)/12 < 5n/12$.

  Now, if $G$ is a planar triangulation, then $|O| = 3$ and $|Y| \leq 1$.
  Thus, from~\eqref{eq:3|X|+|Y|leqn-2+f_4}, 
  we deduce that ${3|S| \leq 3|X|+|Y|+2 \leq n + f_4(H)}$.
  The number of faces in~$G$ is~$2n-4$, and there are~$4f_4(H)$ faces of~$G$ 
  incident to the vertices of degree~4 in~$S$.  
  Therefore, ${f_4(H) \leq (2n-4)/4 < n/2}$.  Hence $|S| < n/2$. 
  By the same argument as before, we conclude that $\id(G) \leq (n+|S|)/4 < 3n/8$.
  Furthermore, if $G$ has minimum degree five, then $f_4(H) = 0$ and $|S| \leq n/3$.
  Hence $\id(G) \leq (n+|S|)/4 \leq n/3$.
\end{proof}

\section{A lower bound}

As far as we know, Theorem~\ref{thm:mainID} might not be tight: 
we do not know a family of planar triangulations $G$ on $n$ vertices 
with $\id(G)$ approaching $3n/8$.
We improve the previous lower bound on \(\eps\)
given by Goddard and Henning~\cite[Figure 6]{GoddardH2020} in the next result.

\begin{theorem}\label{thm:lowerBoundID}
  There is an infinite family $\calF$ of planar triangulations such that
  ${\id(G) = 2n/7}$ for every $G \in \calF$, where $n$ is the number of vertices in $G$.
\end{theorem}

\begin{proof}
Consider the diamond graph depicted in Figure~\ref{fig:27example}(a).
Let us describe a family~$\calF$ of planar triangulations using this graph.
Each planar triangulation in $\calF$ consists of a circular chain of such diamond graphs, 
as depicted in Figure~\ref{fig:27example}(b), with edges added to result in a planar triangulation. 
The planar triangulation $G_k$ obtained in this way with $k$ diamond graphs has $n=7k$ vertices. 
The squared vertices in Figure~\ref{fig:27example}(b) show an independent dominating set with $2k = 2n/7$ vertices.
Note that any independent dominating set in such a planar triangulation $G_k$ must contain at least two vertices in each diamond graph, 
therefore $\id(G_k) = 2k = 2n/7$. 
\end{proof}

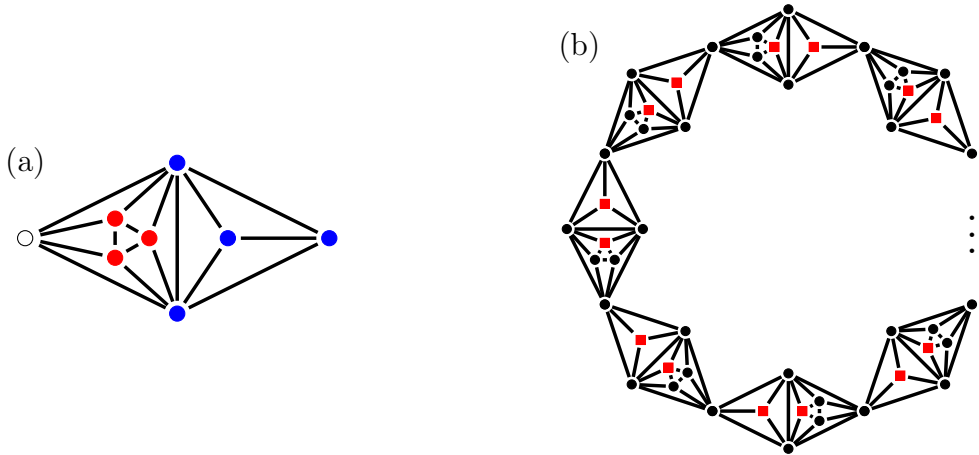
\begin{figure}[h]
\begin{center}
  \parbox{6cm}{
    \begin{tikzpicture}[scale=1]
\node at (0,1) {(a)};
\node (a) [white vertex] at (0,0) {};
\node (b) [blue vertex] at (2,1) {};
\node (c) [blue vertex] at (2,-1) {};
\node (d) [blue vertex] at (4,0) {};

\node (center) [] at ($0.3333*(a) + 0.3333*(b) + 0.3333*(c)$) {};
\node (1) [red vertex] at ($(center) + (120:.3)$) {};
\node (2) [red vertex] at ($(center) + (240:.3)$) {};
\node (3) [red vertex] at ($(center) + (0:.3)$) {};

\node (4) [blue vertex] at ($0.3333*(d) + 0.3333*(b) + 0.3333*(c)$) {};

\draw[line width=1.3pt] (b) -- (c) -- (d) -- (b) -- (a) --(c);
\draw[line width=1.3pt] (a) -- (1) -- (2) -- (a) (b) -- (3) -- (1) --(b) (c) -- (3) -- (2) -- (c);
\draw[line width=1.3pt] (b) -- (4) (c) -- (4) (d) -- (4);
\end{tikzpicture}
  }
  \hspace{1cm}
  \parbox{7cm}{
    \begin{tikzpicture}[scale=.5]
\node at (-3.5,0) {(b)};

\node (ax1) [small black vertex] at (0,0) {};
\node (bx1) [small black vertex] at (2,1) {};
\node (cx1) [small black vertex] at (2,-1) {};
\node (dx1) [small black vertex] at (4,0) {};

\node (centerx1) [] at ($0.3333*(ax1) + 0.3333*(bx1) + 0.3333*(cx1)$) {};
\node (1x1) [small black vertex] at ($(centerx1) + (120:.3)$) {};
\node (2x1) [small black vertex] at ($(centerx1) + (240:.3)$) {};
\node (3x1) [square vertex] at ($(centerx1) + (0:.3)$) {};

\node (4x1) [square vertex] at ($0.3333*(dx1) + 0.3333*(bx1) + 0.3333*(cx1)$) {};

\draw[line width=1.3pt] (bx1) -- (cx1) -- (dx1) -- (bx1) -- (ax1) -- (cx1);
\draw[line width=1.3pt] (ax1) -- (1x1) -- (2x1) -- (ax1) (bx1) -- (3x1) -- (1x1) --(bx1) (cx1) -- (3x1) -- (2x1) -- (cx1);
\draw[line width=1.3pt] (bx1) -- (4x1) (cx1) -- (4x1) (dx1) -- (4x1);

\begin{scope}[shift={(4,0)},rotate=-45]
\node (ax2) [small black vertex] at (0,0) {};
\node (bx2) [small black vertex] at (2,1) {};
\node (cx2) [small black vertex] at (2,-1) {};
\node (dx2) [small black vertex] at (4,0) {};

\node (centerx2) [] at ($0.3333*(ax2) + 0.3333*(bx2) + 0.3333*(cx2)$) {};
\node (1x2) [small black vertex] at ($(centerx2) + (120:.3)$) {};
\node (2x2) [small black vertex] at ($(centerx2) + (240:.3)$) {};
\node (3x2) [square vertex] at ($(centerx2) + (0:.3)$) {};

\node (4x2) [square vertex] at ($0.3333*(dx2) + 0.3333*(bx2) + 0.3333*(cx2)$) {};

\draw[line width=1.3pt] (bx2) -- (cx2) -- (dx2) -- (bx2) -- (ax2) -- (cx2);
\draw[line width=1.3pt] (ax2) -- (1x2) -- (2x2) -- (ax2) (bx2) -- (3x2) -- (1x2) --(bx2) (cx2) -- (3x2) -- (2x2) -- (cx2);
\draw[line width=1.3pt] (bx2) -- (4x2) (cx2) -- (4x2) (dx2) -- (4x2);
\end{scope}

\begin{scope}[shift={(-135:4)},rotate=45]
\node (ax3) [small black vertex] at (0,0) {};
\node (bx3) [small black vertex] at (2,1) {};
\node (cx3) [small black vertex] at (2,-1) {};
\node (dx3) [small black vertex] at (4,0) {};

\node (centerx3) [] at ($0.3333*(ax3) + 0.3333*(bx3) + 0.3333*(cx3)$) {};
\node (1x3) [small black vertex] at ($(centerx3) + (120:.3)$) {};
\node (2x3) [small black vertex] at ($(centerx3) + (240:.3)$) {};
\node (3x3) [square vertex] at ($(centerx3) + (0:.3)$) {};

\node (4x3) [square vertex] at ($0.3333*(dx3) + 0.3333*(bx3) + 0.3333*(cx3)$) {};

\draw[line width=1.3pt] (bx3) -- (cx3) -- (dx3) -- (bx3) -- (ax3) -- (cx3);
\draw[line width=1.3pt] (ax3) -- (1x3) -- (2x3) -- (ax3) (bx3) -- (3x3) -- (1x3) --(bx3) (cx3) -- (3x3) -- (2x3) -- (cx3);
\draw[line width=1.3pt] (bx3) -- (4x3) (cx3) -- (4x3) (dx3) -- (4x3);
\end{scope}

\node (anchor1) [] at ($(-135:4) + (-90:4)$) {};

\begin{scope}[shift={(anchor1)},rotate=90]
\node (ax4) [small black vertex] at (0,0) {};
\node (bx4) [small black vertex] at (2,1) {};
\node (cx4) [small black vertex] at (2,-1) {};
\node (dx4) [small black vertex] at (4,0) {};

\node (centerx4) [] at ($0.3333*(ax4) + 0.3333*(bx4) + 0.3333*(cx4)$) {};
\node (1x4) [small black vertex] at ($(centerx4) + (120:.3)$) {};
\node (2x4) [small black vertex] at ($(centerx4) + (240:.3)$) {};
\node (3x4) [square vertex] at ($(centerx4) + (0:.3)$) {};

\node (4x4) [square vertex] at ($0.3333*(dx4) + 0.3333*(bx4) + 0.3333*(cx4)$) {};

\draw[line width=1.3pt] (bx4) -- (cx4) -- (dx4) -- (bx4) -- (ax4) -- (cx4);
\draw[line width=1.3pt] (ax4) -- (1x4) -- (2x4) -- (ax4) (bx4) -- (3x4) -- (1x4) --(bx4) (cx4) -- (3x4) -- (2x4) -- (cx4);
\draw[line width=1.3pt] (bx4) -- (4x4) (cx4) -- (4x4) (dx4) -- (4x4);
\end{scope}

\node (anchor2) [] at ($(-135:4) + (-90:4) + (-45:4)$) {};

\begin{scope}[shift={(anchor2)},rotate=135]
\node (ax5) [small black vertex] at (0,0) {};
\node (bx5) [small black vertex] at (2,1) {};
\node (cx5) [small black vertex] at (2,-1) {};
\node (dx5) [small black vertex] at (4,0) {};

\node (centerx5) [] at ($0.3333*(ax5) + 0.3333*(bx5) + 0.3333*(cx5)$) {};
\node (1x5) [small black vertex] at ($(centerx5) + (120:.3)$) {};
\node (2x5) [small black vertex] at ($(centerx5) + (240:.3)$) {};
\node (3x5) [square vertex] at ($(centerx5) + (0:.3)$) {};

\node (4x5) [square vertex] at ($0.3333*(dx5) + 0.3333*(bx5) + 0.3333*(cx5)$) {};

\draw[line width=1.3pt] (bx5) -- (cx5) -- (dx5) -- (bx5) -- (ax5) -- (cx5);
\draw[line width=1.3pt] (ax5) -- (1x5) -- (2x5) -- (ax5) (bx5) -- (3x5) -- (1x5) --(bx5) (cx5) -- (3x5) -- (2x5) -- (cx5);
\draw[line width=1.3pt] (bx5) -- (4x5) (cx5) -- (4x5) (dx5) -- (4x5);
\end{scope}

\node (anchor3) [] at ($(ax5) + (4,0)$) {};

\begin{scope}[shift={(anchor3)},rotate=180]
\node (ax6) [small black vertex] at (0,0) {};
\node (bx6) [small black vertex] at (2,1) {};
\node (cx6) [small black vertex] at (2,-1) {};
\node (dx6) [small black vertex] at (4,0) {};

\node (centerx6) [] at ($0.3333*(ax6) + 0.3333*(bx6) + 0.3333*(cx6)$) {};
\node (1x6) [small black vertex] at ($(centerx6) + (120:.3)$) {};
\node (2x6) [small black vertex] at ($(centerx6) + (240:.3)$) {};
\node (3x6) [square vertex] at ($(centerx6) + (0:.3)$) {};

\node (4x6) [square vertex] at ($0.3333*(dx6) + 0.3333*(bx6) + 0.3333*(cx6)$) {};

\draw[line width=1.3pt] (bx6) -- (cx6) -- (dx6) -- (bx6) -- (ax6) -- (cx6);
\draw[line width=1.3pt] (ax6) -- (1x6) -- (2x6) -- (ax6) (bx6) -- (3x6) -- (1x6) --(bx6) (cx6) -- (3x6) -- (2x6) -- (cx6);
\draw[line width=1.3pt] (bx6) -- (4x6) (cx6) -- (4x6) (dx6) -- (4x6);
\end{scope}

\node (anchor4) [] at ($(ax6) + (45:4)$) {};

\begin{scope}[shift={(anchor4)},rotate=225]
\node (ax7) [small black vertex] at (0,0) {};
\node (bx7) [small black vertex] at (2,1) {};
\node (cx7) [small black vertex] at (2,-1) {};
\node (dx7) [small black vertex] at (4,0) {};

\node (centerx7) [] at ($0.3333*(ax7) + 0.3333*(bx7) + 0.3333*(cx7)$) {};
\node (1x7) [small black vertex] at ($(centerx7) + (120:.3)$) {};
\node (2x7) [small black vertex] at ($(centerx7) + (240:.3)$) {};
\node (3x7) [square vertex] at ($(centerx7) + (0:.3)$) {};

\node (4x7) [square vertex] at ($0.3333*(dx7) + 0.3333*(bx7) + 0.3333*(cx7)$) {};

\draw[line width=1.3pt] (bx7) -- (cx7) -- (dx7) -- (bx7) -- (ax7) -- (cx7);
\draw[line width=1.3pt] (ax7) -- (1x7) -- (2x7) -- (ax7) (bx7) -- (3x7) -- (1x7) --(bx7) (cx7) -- (3x7) -- (2x7) -- (cx7);
\draw[line width=1.3pt] (bx7) -- (4x7) (cx7) -- (4x7) (dx7) -- (4x7);
\end{scope}

\node (anchor5) [] at ($0.55*(dx2) + 0.45*(ax7)$) {\Large$\vdots$};

\end{tikzpicture}
  }
\end{center}
\caption{(a) A gadget consisting of seven vertices: the white vertex is not part of the gadget.  
Any independent dominating set has one of the red vertices, otherwise, being independent, 
it cannot dominate the three red vertices. 
Analogously, any independent dominating set has one of the blue vertices,
otherwise it does not dominate the middle blue vertex. 
(b) A graph consisting of a circular chain of gadgets. In each copy of the gadget, two of its vertices 
are needed in any independent dominating set.  The red squared vertices form an independent set with exactly two vertices in each gadget.}
\label{fig:27example}
\end{figure}

\section{Further results and concluding remarks}

In this section, we explore a few families of planar triangulations
for which  we can obtain better bounds on their independent domination number.
A planar \(3\)-tree is a planar triangulation that can be obtained from a triangle
by repeatedly choosing one of its faces and adding a new vertex inside of it while joining this new vertex to
the three vertices of the face.
It is not hard to prove that any 
planar \(3\)-tree admits 
a \(4\)-coloring in which each of its color classes
is dominating. Thus \(\id(G)\leq n/4\) for every planar \(3\)-tree on~\(n\) vertices.

As we observed before, Conjecture \ref{conj:planartr} holds for any Eulerian planar triangulation. 
We can prove a better bound for a particular class of Eulerian triangulations,
which we call \textit{recursive Eulerian triangulations}, and define as follows.
A recursive Eulerian triangulation is either a triangle,
or a graph obtained from a recursive Eulerian triangulation by selecting one of its faces, drawing a triangle inside of it, and joining both ends of each edge of the new triangle with a different vertex of the selected face,
so that the subgraph induced by the chosen face and the new triangle is isomorphic to the octahedron.
It is not hard to check that recursive Eulerian triangulations have the following property.

\begin{proposition}\label{prop:V4triangles}
  In a recursive Eulerian triangulation of order at least 9, 
  the {degree-4} vertices induce a collection of vertex-disjoint triangles.
\end{proposition}

Now, given a graph~$G$, we denote by~$\chi_r(G)$ the minimum 
number of colors in an~$r$-dynamic coloring of~$G$.
Also, fixed a coloring of \(G\), 
we denote by \(L_G(u)\) the set of colors that \(u\) does not see,
i.e., that do not appear in the closed neighborhood of \(u\).

\begin{lemma}\label{lemma:rec-eul-5dyn-6col}
	Let \(G\) be a recursive Eulerian triangulation.
  There is a 5-dynamic 6-coloring of \(G\) such that,
  if \(u\) and \(v\) are adjacent degree-\(4\) vertices, 
  then \(L_G(u) \neq L_G(v)\).
In particular~$\chi_5(G)\leq 6$.
\end{lemma}
\begin{proof}


Let~$G$ be a recursive Eulerian triangulation on~$n$ vertices. We proceed by induction on~$n$.
If~$n\leq 6$ then~$G$ is a triangle or the octahedron and the result is immediate.
So, suppose that~$n\geq 9$.
Let~$abc$ be the last triangle added in the recursive construction of~$G$. 

Let~$G'=G-abc$, and let~$xyz$ be the face of~$G'$ where~$abc$ lies. 
We may assume that~$\{a,x\}, \{b,y\}$ and
$\{c,z\}$ are independent sets.
By induction hypothesis, there exists a 5-dynamic 6-coloring~$\chi' = \{C'_1,C'_2,C'_3,C'_4,C'_5,C'_6\}$ of~$G'$ as above.
For every~$u$, set~$L'(u):=L_{G'}(u)$,~$L(u):=L_G(u)$,~$d'(u):=d_{G'}(u)$ and~$d(u)=d_G(u)$.
Note that~$|L'(u)|=1$ if~$d'(u)=4$ and~$|L'(u)|=0$ if~$d'(u)\geq 6$.
Also, as~$G'$ has at least six vertices, every vertex in~$G'$ has degree at least 4.

Without loss of generality, assume that~$x \in C'_1$,~$y \in C'_2$, and~$z \in C'_3$.
We may also assume that~$L'(x)\subseteq \{4\}$,
$L'(y)\subseteq\{5\}$, and $L'(z)\subseteq\{6\}$.
We define a new coloring~$\chi$ by maintaining the same colors to the vertices in~$G'$ and
assigning color 5 to~$a$, color~6 to~$b$, and color~$4$ to~$c$.
Note that~$\chi$ is in fact a proper coloring for~$G$. 
We claim that~$\chi$ is a coloring as desired.
Indeed, note that~$L(x)=L'(x) \setminus \{4,6\} = \emptyset$.
A similar argument holds for~$y$ and~$z$.
Thus every vertex in~$G$ of degree at least 6 sees all colors.
Also,~$L(a)=\{1\}, L(b)=\{2\}$, and~$L(c)=\{3\}$.
\end{proof}

\begin{theorem}
  For every recursive Eulerian triangulation~$G$ on~$n \geq 9$ vertices, 
  we have~$\id(G) < 13n/42$.
\end{theorem}
\begin{proof}
  Let~$V_4$ be the set of degree-4 vertices of~$G$.

  \begin{claim}\label{claim:6id(G)leqn+V_4}
    $6\,\id(G) \leq n+|V_4|$.
  \end{claim}
  \begin{proof}
    Let~$\chi = \{C_1,C_2,C_3,C_4,C_5,C_6\}$ be a 5-dynamic 6-coloring of~$G$
    as in~Lemma~\ref{lemma:rec-eul-5dyn-6col}. 
    For each~$i$, let~$U_i$ be the set of vertices that are not dominated by~$C_i$. 
    Observe that all vertices with degree at least 5 are dominated by every color of~$\chi$.
    Hence, $\cup_{i=1}^6 U_i$ is precisely the set of degree-4 vertices.
    Moreover, a degree-4 vertex has exactly four different colors on its neighborhood. 
    Thus, the set~$\{U_1,U_2,U_3,U_4,U_5,U_6\}$ partitions~$V_4$.
    Also, by the property of $\chi$ stated in Lemma~\ref{lemma:rec-eul-5dyn-6col}, 
    two adjacent degree-4 vertices are not in the same set $U_i$. 
    Hence, for every $i$, the set~$C_i \cup U_i$ is independent. 
    Therefore, ${6\,\id(G) \leq \sum_{i}|C_i| + \sum_{i}|U_i| = n+|V_4|}$.
  \end{proof}
  
  \begin{claim}\label{claim:V_4leq6n-127}
    $7\,|V_4| \leq 6n-12$.
  \end{claim}
  \begin{proof}
    Let $G'=G-V_4$.
    As $G'$ is a plane graph, the number $f'$ of faces of $G'$ is at most $2|V(G')|-4 = 2(n-|V_4|)-4$.
    By Proposition~\ref{prop:V4triangles}, if there are vertices of $V_4$ inside a face of~$G'$, then 
    there are exactly three such vertices. 
    Hence, $|V_4| \leq 3f' \leq 3(2n-2|V_4|-4)$ and the proof follows.
  \end{proof}
  
  By Claims~\ref{claim:6id(G)leqn+V_4} and~\ref{claim:V_4leq6n-127}, 
  we deduce that $\id(G) \leq \frac{13n-12}{42}$.
\end{proof}


We conclude by observing that, 
if every vertex of a planar triangulation~$G$ on $n$ vertices has odd degree, 
then every color class of a 4-coloring of $G$ is dominating, so $\id(G)\leq n/4$. 
We can extend this result and show that if $G$ has at least $\alpha n$ odd-degree vertices,
then $\id(G)\leq (2-\alpha)n/4$, which improves the bound in Theorem~\ref{thm:mainID} when $\alpha \geq 2/7$.
Also, Conjecture~\ref{conj:planartr} holds for any \(n\)-vertex planar triangulation 
with at least \(2n/3\) odd-degree vertices. 




\bibliographystyle{plain}
\bibliography{domset.bib}

\begin{thebibliography}{10}

\bibitem{Borodin1979}
Oleg~V. Borodin.
\newblock On acyclic colorings of planar graphs.
\newblock {\em Discrete Math.}, 25(3):211--236, 1979.

\bibitem{Campos2013}
Christiane~N. Campos and Yoshiko Wakabayashi.
\newblock On dominating sets of maximal outerplanar graphs.
\newblock {\em Discrete Appl. Math.}, 161(3):330--335, 2013.

\bibitem{Diestel2010}
Reinhard Diestel.
\newblock {\em Graph Theory}, volume 173 of {\em Graduate Texts in
  Mathematics}.
\newblock Springer, Heidelberg, fourth edition, 2010.

\bibitem{Goddard2013}
Wayne Goddard and Michael~A. Henning.
\newblock Independent domination in graphs: a survey and recent results.
\newblock {\em Discrete Math.}, 313(7):839--854, 2013.

\bibitem{GoddardH2020}
Wayne Goddard and Michael~A. Henning.
\newblock Independent domination, colorings and the fractional idomatic number
  of a graph.
\newblock {\em Appl. Math. Comput.}, 382:125340, 8, 2020.

\bibitem{Goddard2023}
Wayne Goddard and Michael~A. Henning.
\newblock Independent domination in outerplanar graphs.
\newblock {\em Discrete Applied Mathematics}, 325:52--57, 2023.

\bibitem{KingP2010}
Erika L.~C. King and Michael~J. Pelsmajer.
\newblock Dominating sets in plane triangulations.
\newblock {\em Discrete Math.}, 310(17-18):2221--2230, 2010.

\bibitem{LiZSX2016}
Zepeng Li, Enqiang Zhu, Zehui Shao, and Jin Xu.
\newblock On dominating sets of maximal outerplanar and planar graphs.
\newblock {\em Discrete Appl. Math.}, 198:164--169, 2016.

\bibitem{MathesonT1996}
Lesley~R. Matheson and Robert~E. Tarjan.
\newblock Dominating sets in planar graphs.
\newblock {\em European J. Combin.}, 17(6):565--568, 1996.

\bibitem{Plummer2016}
Michael~D. Plummer, Dong Ye, and Xiaoya Zha.
\newblock Dominating plane triangulations.
\newblock {\em Discrete Appl. Math.}, 211:175--182, 2016.

\bibitem{Plummer2020}
Michael~D. Plummer, Dong Ye, and Xiaoya Zha.
\newblock Dominating maximal outerplane graphs and {H}amiltonian plane
  triangulations.
\newblock {\em Discrete Appl. Math.}, 282:162--167, 2020.

\bibitem{Tokunaga2013}
Shin-ichi Tokunaga.
\newblock Dominating sets of maximal outerplanar graphs.
\newblock {\em Discrete Appl. Math.}, 161(18):3097--3099, 2013.

\bibitem{Spacapan2020}
Simon \v{S}pacapan.
\newblock The domination number of plane triangulations.
\newblock {\em J. Combin. Theory Ser. B}, 143:42--64, 2020.

\bibitem{ZvervichZ1995}
Igor~E. Zverovich and Vadim~E. Zverovich.
\newblock An induced subgraph characterization of domination perfect graphs.
\newblock {\em J. Graph Theory}, 20(3):375--395, 1995.

\end{thebibliography}

\end{document}